\documentclass[12pt]{article}
\usepackage{e-jcCHANGE}

\usepackage{amsthm,amsmath,amssymb}

\usepackage{graphicx}
\usepackage{tikz}
\usetikzlibrary{topaths,calc}

\usepackage[colorlinks=true,citecolor=black,linkcolor=black,urlcolor=blue]{hyperref}

\theoremstyle{plain}
\newtheorem{theorem}{Theorem}[section]
\newtheorem{lemma}[theorem]{Lemma}
\newtheorem{corollary}[theorem]{Corollary}

\theoremstyle{definition}

\theoremstyle{remark}
\newtheorem{remark}[theorem]{Remark}

\title{\bf On Weak Hamiltonicity of a Random Hypergraph}

\author{Daniel Poole\thanks{The author gratefully acknowledges support from NSF grant \# DMS-1101237.}\\
\small Department of Mathematics\\[-0.8ex]
\small The Ohio State University\\[-0.8ex] 
\small Columbus, Ohio, U.S.A.\\
\small\tt poole@math.osu.edu
}

\begin{document}

\maketitle


\begin{abstract}
A {\it weak (Berge) cycle} is an alternating sequence of vertices and (hyper)edges $C=(v_0, e_1, v_1, \ldots, v_{\ell-1}, e_\ell, v_{\ell}=v_0)$ such that the vertices $v_0, \ldots, v_{\ell-1}$ are distinct with $v_k, v_{k+1} \in e_{k}$ for each $k$, but the edges $e_1, \ldots, e_\ell$ are not necessarily distinct. We prove that the main barrier to the random $d$-uniform hypergraph $H_d(n,p),$ where each of the potential edges of cardinality $d$ is present with probability $p$, developing a weak Hamilton cycle is the presence of isolated vertices. In particular, for $d \geq 3$ fixed and $p=(d-1)! \frac{\ln n + c}{n^{d-1}}$, the probability that $H_d(n, p)$ has a weak Hamilton cycle tends to $e^{-e^{-c}}$, which is also the limiting probability that $H_d(n,p)$ has no isolated vertices. As a consequence, the probability that the random hypergraph $H_d(n, m=\frac{n(\ln n + c)}{d}),$ where $m$ potential edges are chosen uniformly at random to be present, is weak Hamiltonian also tends to $e^{-e^{-c}}$.

\end{abstract}

\section{Introduction}

A $d$-uniform hypergraph is a pair $(V, E)$ of vertices $V$ and (hyper)edges $E$, where $E \subseteq { V \choose d}$. Let $H_d(n,p)$ denote the random $d$-uniform hypergraph with vertex set $[n]:=\{1, 2, \ldots, n\}$, where each of the ${n \choose d}$ potential edges is present with probability $p$, independently of all other potential edges. The similar model $H_d(n, m)$ is a random $d$-uniform hypergraph on $[n]$ where $m$ edges are chosen uniformly at random among all sets of $m$ potential edges.  For $d=2$, these models are the typical random graph models, $G(n,p)$ and $G(n,m).$ 

As customary, we say that for a given $p=p(n)$ ($m$ resp.) some graph property holds {\it with high probability}, denoted {\it w.h.p.}, if the probability that $H_d(n,p)$ ($H_d(n,m)$ resp.) has this property tends to 1 as $n \to \infty$.

Existence problems of Hamilton cycles have a rich history in random graphs. First posed by Erd\H{o}s and R\'{e}nyi \cite{ER61a}, even a correct threshold of Hamiltonicity resisted the efforts of researchers until a breakthrough by P\'{o}sa \cite{posa} and Korshunov \cite{korshunov}, who found that w.h.p. $G(n,m=c n \ln n)$ has a Hamilton cycle if $c$ is sufficiently large. Later Korshunov \cite{korshunov2}, Koml\'{o}s and Szemer\'{e}di \cite{komlos}, and Bollob\'{a}s \cite{bol83} established the sharp theshold for Hamiltonicity as well as the probability of Hamiltonicity within the narrow window enclosing the critical $p$ (resp. $m$) in $G(n,p)$ (resp. $G(n,m)$); in particular, if $p=\frac{\ln n + \ln \ln n + c}{n}$, then the probability that $G(n,p)$ is Hamiltonian tends to $e^{-e^{-c}}$. The probability that each vertex in $G(n,p)$ has degree at least 2, a trivial necessary condition for Hamiltonicity, also tends to this limit.  As a culmination point, Bollob\'{a}s \cite{Bol 84} showed w.h.p. a Hamilton cycle is born at the first moment that the random graph process, where edges are added uniformly at random one after another, has minimum degree at least $2$. 

As for hypergraphs, the first problem that arises is how to even define a cycle. There has been many fine results on the so-called $\ell$-overlapping cycles. See K\"{u}hn and Osthus~\cite{kuhn} for a great survey of results, and in particular, see Dudek and Frieze~\cite{dudek tight}, where they establish the threshold values for the appearance of many of these cycles. However, we use a different twist on the classical notion of a hypergraph cycle, defined by Berge in~\cite{berge}.

For $\ell \geq 3$, a {\it (Berge) cycle, $C=(v_0, e_1, v_1, \ldots, e_\ell, v_\ell=v_0),$ of length $\ell$} is an alternating sequence of vertices and edges such that $v_1, \ldots, v_{\ell-1}$ are distinct, $e_1, \ldots, e_\ell$ are distinct, and $v_{i-1},v_i \in e_i.$ A {\it weak} cycle is defined similarly except we do not require the edges to be distinct; see Figure \ref{fig: weak H cycle} for an example. For the graph case $(d=2)$, these edges must be distinct, and thus weak cycles are cycles. In proper hypergraphs ($d \geq 3$), one crucial difference between cycles and weak cycles is that a vertices with degree 1 can be in weak cycles, but necessarily these  vertices can not be in cycles.

\begin{figure}[ht]
\centering
\begin{tikzpicture}[scale=.6]
\node (v1) at (2,2) {};
\node (v2) at (4, 2) {};
\node (v3) at (6, 0) {};
\node (v4) at (4, -2) {};
\node (v5) at (2, -2) {};
\node (v6) at (0, -1) {};
\node (v7) at (0, 1) {};
\node (v8) at (6, 2.5) {};
\node (v9) at (-1,-1) {};
\node (v10) at (-2,-.5) {};
\node (v11) at (-.5, -1.9) {};
\draw[dashed] (2,2) -- (4,2) -- (6, 0) -- (4, -2) -- (2, -2) -- (0, -1) -- (0, 1) -- (2, 2);
\begin{scope}[fill opacity=0]
\draw plot [smooth cycle] coordinates {($(v6) +(0.2,-.2)$) ($(v11)+(.2,-.2)$) ($(v11)+(-.2,-.1)$) ($(v9)+(-.2,-.2)$) ($(v10)+(-.2, -.2)$) ($(v10)+(-.05,.2)$) ($(v10)+(.5, .3)$) ($(v6)+(.2,.2)$)};
\draw plot [smooth cycle, tension=.2] coordinates {($(v1)+(-.2,.3)$) ($(v2)+(.2,.3)$) ($(v4)+(.2,-.3)$) ($(v5)+(-.2,-.3)$)};
\draw plot [smooth cycle, tension=.5] coordinates {($(v2)+(-.2,.2)$) ($(v2)+(1,.1)$) ($(v8)+(.2,.2)$) ($(v3)+(.2,0)$) ($(v4)+(-.2,-.2)$)};
\draw plot [smooth cycle, tension=.4] coordinates {($(v1)+(.2,.2)$) ($(v5)+(.2,-.2)$) ($(v6)+(-.1,-.1)$) ($(v7)+(-.1,.1)$)};
\end{scope}
\fill (v1) circle (0.1);
\fill (v2) circle (0.1);
\fill (v3) circle (0.1);
\fill (v4) circle (0.1);
\fill (v5) circle (0.1);
\fill (v6) circle (0.1);
\fill (v7) circle (0.1);
\fill (v8) circle (0.1);
\fill (v9) circle (0.1);
\fill (v10) circle (0.1);
\fill (v11) circle (0.1);
\end{tikzpicture}
\caption{A weak cycle (dashed curve) in a $4$-uniform hypergraph}
\label{fig: weak H cycle}
\end{figure}
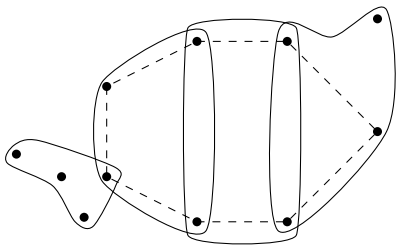

We'll say that the weak cycle $C$ {\it spans} the vertices $v_0, v_1, \ldots, v_{\ell-1}$; note that the edges making up $C$ possibly contain more than just the spanned vertices (again, see Figure \ref{fig: weak H cycle}). A {\it weak Hamilton cycle} is a cycle spanning the entire vertex set, and a hypergraph is {\it weak Hamiltonian} if it contains a weak Hamilton cycle. 

Trivially, each vertex in a weak Hamiltonian hypergraph must have degree at least 1. However, we found that the main barrier to weak Hamiltonicity in $H_d(n,p)$ is this ``local" obstruction about whether or not this hypergraph has this trivial necessary condition.

\begin{theorem}\label{theorem: 10.6.1}
\textbf{(i)} For $d \geq 3$ and $p=(d-1)! \tfrac{\ln n + c_n}{n^{d-1}}$, where $c_n \to c \in \mathbb{R}$,
\begin{align*}
P(H_d(n,p) \text{ is weak Hamiltonian}) = P(\min \mathrm{deg}(H_d(n,p)) \geq 1)+o(1) \to e^{-e^{-c}}.
\end{align*}
\textbf{(ii)} If $m=\tfrac{n}{d} \left( \ln n + c_n\right),$ then
\begin{align*}
P(H_d(n,m) \text{ is weak Hamiltonian}) & = P(\min \mathrm{deg}(H_d(n,m)) \geq 1)+o(1)\to e^{-e^{-c}}.
\end{align*}
Consequently, for any $c_n \to -\infty$, w.h.p $H_d(n,p)$ and $H_d(n,m)$ are not weak Hamiltonian and for any $c_n \to \infty$, w.h.p. $H_d(n,p)$ and $H_d(n,m)$ are weak Hamiltonian. 
\end{theorem}
First, since weak Hamiltonicity is an increasing graph property, part \textbf{(ii)} can be easily shown to follow from  \textbf{(i)} using a standard random graphs argument. As such, we omit the proof of \textbf{(ii)}.

\begin{remark}
This theorem appears to be in contrast with the Hamiltonicity results of $G(n,p)$ and $G(n,m)$, where the main barrier is the presence of vertices of degree less than 2. Although this seems like a behavioral difference between the graph and the hypergraph models, notice that a vertex in a hypergraph ($d \geq 3$) with degree at least 1 already has at least 2 neighbors; so in both random graphs and hypergraphs, the main barrier for weak Hamiltonicity is the presence of vertices with less than 2 neighbors. 
\end{remark}

\begin{remark}
Having found the threshold of weak Hamiltonicity of $H_d(n,m)$ and $H_d(n,p)$, we naturally wondered whether the $d$-uniform hypergraph process, where edges are added to $V=[n]$ uniformly at random one after another, satisfies a similar result. That is, we believe that w.h.p. the moment this process loses its last isolated vertex, denoted $\tau_d$, is also the moment that the hypergraph becomes weak Hamiltonian, denoted $T_d$. Trivially, $\tau_d \leq T_d$, and by our theorem, w.h.p. $T_d - \tau_d = o(n)$. We feel that our method here could be significantly refined to show that $T_d-\tau_d=O_p(\ln n)$, but by its nature will not be able to prove our suspected result that w.h.p. $\tau_d=T_d$. If true, this conjecture would generalize to hypergraphs the weak Hamilton version of Bollob\'{a}s'~\cite{Bol 84} hitting time result. 
\end{remark}

\subsection{Sketch of the proof}

We actually prove a stronger result than Theorem \ref{theorem: 10.6.1}, which says that for any $p$ in the critical range, $H_d(n,p)$ has as large a weak cycle as can be expected. 
\begin{theorem}
Let $c_n \to c$ and $p=(d-1)! \frac{\ln n + c_n}{n^{d-1}}$. W.h.p. $H_d(n,p)$ has a weak cycle spanning all non-isolated vertices.
\end{theorem}
First, this theorem establishes the first equality in statement \textbf{(i)} of Theorem \ref{theorem: 10.6.1}. To get a handle on the probability that there are no isolated vertices, we have the following lemma. 
\begin{lemma}\label{lemma: converge in distribution}
Suppose $c_n \to c \in \mathbb{R}$. Let $X_n$ denote the number of isolated vertices in $H_d(n, p=(d-1)! \frac{\ln n + c_n}{n^{d-1}})$. Then $X_n$ converges in distribution to a Poisson random variable with mean $e^{-c}$. In particular, 
\begin{equation*}
P(H_d(n,p) \text{ has no isolated vertices}) \to e^{-e^{-c}}.
\end{equation*}
\end{lemma}
\begin{remark}
One can prove this lemma using standard techniques by computing the factorial moments of $X_n$ and finding that $E[(X_n)_r] \sim (e^{-c})^r$ for any fixed $r=1,2,\ldots.$ As such, the proof is omitted.
\end{remark}

To find this sharp threshold, we will follow a hypergraph version of Bollob\'{a}s' proof (see  \cite{BB}) that w.h.p. $G(n, p=\frac{\ln n + \ln \ln n + \omega}{n})$ is Hamiltonian with a few key differences. We outline the proof of the theorem here, hoping that it will help following the later arguments. 
\begin{enumerate}
\item  In section \ref{sec: de la vega}: Via an analogue of de la Vega's Theorem, it will be shown that there is likely a path of length $n-o(n)$ for $p$ sufficiently close to but less than the sharp threshold of connectivity. 
\item In section \ref{sec: posa}: We obtain a necessary condition for an extremal path starting at a fixed vertex, which actually says that the set of endpoints of all paths obtained via rotations is ``non-expanding." 
\item In section \ref{sec: nonexpanding}: Then, we show that all non-expanding sets are sufficiently large, which means that the set of endpoints of rotations of a longest path must be large as well. This will imply that for any such graph, there are sufficiently many non-present edges (on the order $n^d$) such that the addition of any one of these absent edges will connect two endpoints of a longest path creating a cycle. As long as this cycle is not isolated, we can break apart the cycle to create a longer path than before.
\item In section \ref{sec: increment}: By increasing $p$ in small increments, we will be assured of likely adding one of these beneficial edges from step (3). We keep increasing our edge probability until we finally end with a weak cycle on the set of non-isolated vertices. 
\end{enumerate}

The first, and most obvious, difference from Bollob\'{a}s proof, is that certain enumerative issues that are relatively simple for graphs become much more difficult when switching to hypergraphs. For instance, in the proof that all non-expanding sets must be large, we encounter the probability that each vertex of a set of $b$ vertices, denoted $B$, is adjacent to at least one of some $a$ vertices, denoted $A$, using only edges contained within $A \cup B$. For $G(n,p)$, this probability is precisely $(1-q^a)^b$, since for different vertices of $B$, the events that these vertices are adjacent to $A$ are independent. However, for the hypergraph case, these events are definitely not independent, and we introduce a greedy edge finding algorithm, which we analyze to establish a useful bound on this probability (that we feel is interesting in its own right). 

Second, we are also concerned with $p$ in the critical range, where the limiting probability that $H_d(n,p)$ is weak Hamiltonian is strictly between 0 and 1. In this range, there are possibly isolated vertices, and we need to be more careful about the lack of ``small" non-expanding sets. Further, when we increase our edge probability $p$ in small increments, we will end up being forced to do it in 2 large steps rather than just 1.

\section{Hypergraph Analogue of de la Vega's Theorem}\label{sec: de la vega}

Just prior to the typical window where the random hypergraph becomes connected, we will show that there is a path of length $n-o(n)$. An analogous result, established by Fernandez de la Vega~\cite{de la vega}, was used to determine the sharp threshold of Hamiltonicity in $G(n, p)$. Here, we prove a hypergraphic version of de la Vega's theorem, which in fact is established by the graph version after a key reduction argument. The following statement of de la Vega's Theorem is in Bollob\'{a}s~\cite{BB}.

\begin{theorem}[de la Vega's Theorem]\label{de la vega}
Let $\Theta= \Theta(n) \in \left(4 \ln 2 , \ln n - 3 \ln\ln n\right)$, then w.h.p. there is a path in $G(n, p = \Theta/n)$ of length at least $n \left( 1 - \frac{4 \ln 2}{\Theta} \right) $.
\end{theorem}

We will be using the following immediate corollary of de la Vega's Theorem.

\begin{corollary}\label{cor de la vega}
Let $\Theta \in \left(4 \ln 2 , \frac{1}{2}\ln n \right)$, then with high probability, there is a path in $G\left(\lfloor n/2 \rfloor, p = \Theta / \lfloor n/2  \rfloor \right)$ of length at least 
\begin{equation*}
\frac{n}{2} \left( 1 - \frac{3}{\Theta} \right).
\end{equation*}
\end{corollary}

In the following hypergraph analogue of de la Vega's Theorem, we do not try to obtain the best possible bounds on a likely long path, instead we care only to find sufficient bounds that we will later use to prove our ultimate result. 

\begin{lemma}\label{hyper de la vega}
Suppose $\sigma(n) \to \infty$ as $n \to \infty$ so that $\sigma = o(\ln n)$. Let $\theta = \theta(n) \in [\sigma, 2 \ln n]$ and $p= (d-1)! \frac{\theta}{n^{d-1}}.$ W.h.p. in $H_d\left(n, p \right)$, there is a path of length at least $n - \frac{2^{d+2}}{\theta} n$. 
\end{lemma}
\begin{proof}
Our essential argument is that we decompose $H_d(n,p)$ into 3 (random) hypergraphs $H^1$, $H^2$, and $H^3$. First, on $H^1$, we will prove there is a path of length approximately $n/2$ along only the first $n/2$ vertices; second,  we prove that in $H^2$, there is a path along only the last $n/2$ vertices;   finally, using the edges of $H^3$, we concatenate these long paths together.

Let $p_1 =p/3$ and let $H^{i}$, for $i=1,2,3$, be random hypergraphs on $[n]$, which are {\it independent} and distributed as $H_d (n, p_1)$. Note that $H:=H^1 \cup H^2 \cup H^3$ is distributed as $H_d(n, p')$, where the probability that a generic edge in $H$ is present, $p'$, is equal to the probability that at least one of $H^1, H^2, H^3$ have this edge; in particular $p'=1-(1-p_1)^3$. Moreover, $p' \leq 3 p_1 =p$ and so $H$ can be naturally coupled with $H_d(n,p)$ such that $H$ is a subgraph of $H_d(n,p)$. Therefore, it suffices to show that w.h.p. $H$ has a path of the desired length.

Let's begin with finding a long path in $H^1$ along only the first $n/2$ vertices. We construct a random graph, $G^1$, with vertex set $\big [ \lfloor n/2 \rfloor \big ]$ by the following: for each possible edge $\{i, j \} \subset \big[ \lfloor n/2 \rfloor \big ], i \neq j$, we define $\{i,j\} \in E(G^1)$ if and only if there is some hyperedge in $H^{1}$ that contains $i$ and $j$ but no other vertex of $\big[ \lfloor n/2 \rfloor \big]$, i.e. 
\begin{equation*}
\{i,j\} \in E(G^1) \text{ iff } \{i, j, v_3, \dots, v_d\} \in E(H^1) \text{ for some } \{v_3, \ldots, v_d\} \subset \{ \big \lfloor n/2 \big \rfloor +1, \dots, n \big \}.
\end{equation*}
Each potential edge in $G^1$ considers distinct potential hyperedges in $H^{1}$, and thus the potential edges in $G^1$ are present independently of one another.  Further, the potential edge $\{i,j\}$ is not in $G^1$ iff all ${\lceil n/2 \rceil \choose d-2}$ potential hyperedges of the form $\{i, j, w_3, \ldots, w_d\}$ are not present. Hence 
\begin{align*}
P( \{i,j\} \notin E ) =  (1 - p_1)^{ { \lceil n/2 \rceil\choose d-2} }=: 1 - p^*.
\end{align*}
Consequently, $G^{1}$ is equal in distribution to $G\left( \big \lfloor n/2 \big \rfloor, p^* \right)$ and
\begin{equation*}
p^*  = 1 - e^{-p_1 { \lceil n/2 \rceil \choose d-2} + O\left( p_1^2 n^{d-2}\right) } = \frac{d-1}{3\cdot 2^{d-2}} \frac{\theta}{n}  + O\left( \frac{(\ln n)^2}{n^2} \right).
\end{equation*}
For large enough $n$, we have that 
\begin{equation*}
p^* \geq 0.9 \frac{d-1}{3 \cdot 2^{d-1}} \frac{\theta}{\lfloor n/2 \rfloor}.
\end{equation*}
In particular, we meet the conditions to apply the corollary to de la Vega's theorem (Corollary \ref{cor de la vega}) to $G^1$; w.h.p. there is a path in $G^1$ of length at least 
\begin{equation*}
\frac{n}{2} \left( 1 - \frac{10 \cdot 2^{d-1}}{(d-1) \theta} \right),
\end{equation*} 
which corresponds to a path in $H^1$ of same length spanning only vertices in $\big \lfloor n/2 \big \rfloor$ (note that the hyperedges that make up this path will include larger index vertices).

Now for finding long paths in $H^2$, we construct another random graph, $G^2$, on $\big[ \lfloor n/2\rfloor \big]$ by the following: for each possible edge $\{i, j \} \subset \big [ \lfloor n/2 \rfloor \big]$, we define $\{i,j\} \in E(G^2)$ if and only if there is some hyperedge in $H^{2}$ of the form $\big\{ \lceil \frac{n}{2} \rceil+  i, \lceil \frac{n}{2} \rceil+  j,v_3, \dots, v_d \big\} $ where $\{v_3, \dots, v_d\} \subset \big[ \lceil n/2 \rceil \big]$. By the same argument as before, w.h.p. there is a path in $G^2$ of length at least 
\begin{equation*}
\frac{n}{2} \left( 1 - \frac{10 \cdot 2^{d-1}}{(d-1) \theta} \right),
\end{equation*}
which corresponds to a path in $H^2$ of the same length spanning only vertices in $\big\{ \lceil n/2 \rceil +1, \cdots, n \big\}$.

Using $H^{3}$, we will concatenate these paths together. Let $\mathcal{A}$ denote the event that these long paths exists in $H^{1}$ and $H^{2}$. On $\mathcal{A}$, let $A$ be the last $n/ \ln n$ vertices of one such long path in $H^{1}$ and $B$ be the first $n/ \ln n$ vertices of one such long path in $H^{2}$ (both of these are defined because the length of these paths is $\frac{n}{2} -o(n) \gg n/\ln n$). 

By the independence of potential hyperedges of $H^3$ from $H^1$ and $H^2,$
\begin{equation*}
P(\mathcal{A} \cap \{ \text{no hyperedge between }A,B\}) \leq \left(1 - p_1 \right)^{ \left(\frac{n}{\ln n}\right) \left(\frac{n}{\ln n}\right)  { n - 2 \frac{n}{\ln n} \choose d-2} } \leq e^{-p_1 \frac{n^2}{(\ln n)^2} \frac{n^{d-2}}{(d-2)!} + o(1)},
\end{equation*}
which tends to zero. Hence, with high probability these two long paths on distinct vertices exist and there is a hyperedge that contains one of the last $n/\ln n$ vertices, say $a$, of the first path and one of the first $n/\ln n$ vertices, say $b$, of the second path. We construct our path in $H$ by following the long path in $H^{1}$ until we reach the vertex $a$, then choosing this connecting hyperedge containing $a$ and $b$, then following the path in $H^{2}$ until it ends. This path has length at least 
\begin{equation*}
 2 \left( \frac{n}{2} \left( 1 - \frac{10 \cdot 2^{d-1}}{(d-1) \theta} \right) - \frac{n}{\ln n} \right)+ 1  \geq n - \frac{10 \cdot 2^{d-1}}{\theta(d-1)} n - 2 \frac{n}{\ln n} \geq n - \frac{2^{d+2}}{\theta} n,
 \end{equation*}
as desired.
\end{proof}

\section{P\'{o}sa's Lemma}\label{sec: posa}

Now that we have established that paths of length $n-o(n)$ are likely to exist near the connectedness threshold of $H_d(n,p)$, our next goal is to show that the sets of endpoints of longest paths are relatively large (on the order of $n$). In the proof of finding the Hamiltonicity threshold of $G(n,p)$ (see~\cite{BB}), P\'{o}sa's Lemma is indispensable in describing the set of endpoints of these longest paths; in particular, the set of endpoints is ``non-expanding". It plays a similar role here as well.

We begin with a couple of definitions. For a hypergraph, $H$, on $[n]$, and a set of vertices $V$, we define $N(V)$ as the set of neighbors of $V$. Formally, 
\begin{equation*}
N(V) = \{ w \in [n] \setminus V: \exists e \in E(H), \exists v \in V \text{ such that } w, v \in e\}
\end{equation*}
We say that a set of vertices, $A$, is {\it non-expanding} if  $|N(A)|<2|A|.$

The hypergraph version of P\'{o}sa's Lemma is effectively the same as the graph version. Let $H$ be a hypergraph and let $P = (v_0, e_1, v_1, \ldots, e_h, v_h)$ be a longest path starting from $v_0$. Note that any neighbor of an endpoint of a longest path must necessarily be in said path or else we could extend this supposedly longest path and reach a contradiction.  Now suppose there is some present hyperedge $e$ containing both $v_h$ and $v_i$ for some $i<h$. We say that the path 
\begin{equation*}
P' = (v_0, e_1, v_1, \ldots, e_i, v_i, e, v_h, e_{h-1}, v_{h-1}, \ldots, e_{i+1}, v_{i+1})
\end{equation*}
is a {\it rotation} of $P$ by $\{ v_i, v_h\}$; see Figure \ref{fig:InformativeFigure} for an illustration of these paths.
\begin{figure}[!h]
\begin{center}
\begin{tikzpicture}[scale=.7]
\node at (-5,.5) {Before rotation};
\draw (0,0) -- (1,1) -- (2,0) -- (3,1) -- (4,0) -- (5,1) -- (6,0) -- (7,1) -- (8,0) -- (9,1) -- (10,0);
\draw[dashed] (10,-0.4) to[out=270, in=270] (4,-0.4);
\node at (0,-.2) {$v_0$};
\node at (1,1.2) {$v_1$};
\node at (5,1.2) {$v_{i+1}$};
\node at (4,-.2) {$v_i$};
\node at (10,-.2) {$v_{h}$};
\node at (-1,.5) {$P:$};
\node at (-5, -3.5) {After rotation};
\draw (0,-4) -- (1, -3) -- (2, -4) -- (3, -3) -- (4, -4);
\draw (10,-4.4) to [out=270, in=270] (4, -4.4);
\draw (5,-3) -- (6,-4) -- (7,-3) -- (8,-4) -- (9,-3) -- (10,-4);
\node at (0,-4.2) {$v_0$};
\node at (1,-2.8) {$v_1$};
\node at (4,-4.2) {$v_i$};
\node at (5,-2.8) {$v_{i+1}$};
\node at (10,-4.2) {$v_{h}$};
\node at (-1, -3.5) {$P':$};
\node at (7,-5.7) {$e$};
\end{tikzpicture}
\end{center}
\caption{\label{fig:InformativeFigure} Rotation of a path $P$.}
\end{figure}
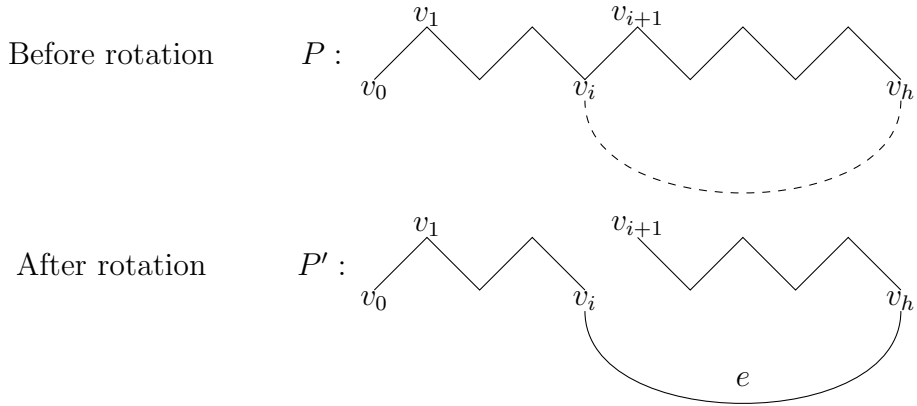

Note that the vertex set of the path is unchanged after rotation. As opposed to graphs, each hyperedge may give rise to more than one such rotation for a given path; in fact, $e$ could have already been present within the path $P$! 

The hypergraph version of P\'{o}sa's Lemma requires that we allow the possibility of repeated hyperedges in the paths which in turn gives rise to our definition of a {\it weak} cycle allowing duplicate hyperedges as well. In fact, this is the precise moment in which we turned to {\it weak} cycles.

To proceed, let $\mathcal{P}(P, v_0)$ be the set of paths obtained by the path $P$, starting at $v_0$ via any number of rotations and let $S = S(P,v_0)$ be the set of all endpoints of $\mathcal{P}(P,v_0)$ excluding $v_0$. This set $S$ is referred to as a {\it P\'{o}sa set}. 

\begin{lemma}[P\'{o}sa's Lemma]\label{posa}
$S$ is non-expanding. In other words, $|N(S)|<2|S|.$
\end{lemma}

By allowing repeated edges as we do, the proof of P\'{o}sa's Lemma for graphs can be naturally extended to hypergraphs. In fact, the proof of P\'{o}sa's Lemma in Bollob\'{a}s~\cite{BB} can be followed word for word here. Consequently, we omit the proof.

In order to show that the set of endpoints formed by rotations of a longest path must likely be large, we will show that any ``non-trivial" non-expanding set is large. Since any set of isolated vertices trivially form a non-expanding set, we will focus on non-expanding sets of non-isolated vertices. For a hypergraph $H$, let $u(H)$ be the size of the smallest non-expanding set of non-isolated vertices. Formally, if $V_1(H)$ denotes the set of non-isolated vertices of $H$, then 
\begin{equation}\label{def of uH}
u(H) := \max \{ u : \text{if } A \subset V_1(H), |A| < u, \text{then } |N(A)| \geq 2|A| \}.
\end{equation}
As a consequence of P\'{o}sa's Lemma, for any longest path, $P$, starting at non-isolated $v_0$, we must have that $|S(P, v_0)| \geq u(H)$. An important fact that we'll later use is that if $H$ is connected, then the addition of edges to $H$ can only increase this measure $u$. 

Our eventual goal is to show that in $H_d(n,p),$ for $p$ near the connectivity threshold, w.h.p. $u$ is on the order of $n$. A very useful corollary of P\'{o}sa's Lemma for graphs, due to Bollob\'{a}s~\cite{BB}, gives a lower bound on the number of absent edges whose addition necessarily extends a longest path.  We give an analogous hypergraph version. 

\begin{corollary}\label{absent}
Let $P$ be a longest path in hypergraph, $H,$ on $[n]$, and suppose the length of this path is $h$. If $H$ does not have a weak cycle of length $h+1$, then there are at least 
\begin{equation*}
\frac{ u(H) \left( {n-1\choose d-1} - {n-1-u(H) \choose d-1} \right)}{d} 
\end{equation*}
absent edges of $H$ such that the addition of any one of these non-present edges creates a weak cycle of length $h+1$.
\end{corollary}

\begin{proof}
Let $v_0$ be one of the two endpoints of $P$ and $u=u(H)$. We define the set of endpoints formed by rotations of $P$ by $S$ as before. By P\'{o}sa's Lemma, we know that $S$ is non-expanding and so $|S| \geq u$. Let $w_1, w_2, \cdots, w_{u}$ be distinct vertices of $S$ with some largest paths $P_1,$ $P_2,$ $\cdots,$ $P_{u}$ starting at $v_0$ that end in $w_1,$ $w_2,$ $\cdots,$ $w_{u}$, respectively. Now for each $i \leq u$, let $S_i$ be the set of endpoints formed by rotations of a longest path $P_i$ starting from $w_i$; in other words, let $S_i = S(P_i, w_i).$  Again by P\'{o}sa's Lemma, we have that $|S_i| \geq u$ and there are some $t^i_1, t^i_2, \cdots t^i_{u}$ vertices in  $S_i$. If there is some edge containing $w_i$ and at least one vertex of $S_i$, then there would be a weak cycle of length $h+1$; by hypothesis, no such cycle exists. To finish off the proof, we enumerate these missing edges. 

Let $\mathcal{E}_i$ be the collection of absent edges $e$ which contain $w_i$ and some $t^i_j$. We wish to determine a lower bound on $|\mathop{\cup}_i \mathcal{E}_i|$. Note that
\begin{align*}
|\mathcal{E}_i| &={u \choose 1}{n-1-u \choose d-2} + {u \choose 2}{n-1-u \choose d-3} + \cdots + {u \choose d-1}{n-1-u \choose 0} \\&= {n-1 \choose d-1} - {n-1-u \choose d-1}.
\end{align*}
Further, each absent edge of $\mathcal{E}_i$ contains $w_i$, and so any of these absent edges can be in at most $d$ different $\mathcal{E}_i$'s. Therefore
\begin{equation*}
\left| \mathop{\cup}_{i=1}^{u} \mathcal{E}_i \right|  \geq \frac{ u \left( {n-1 \choose d-1}-{n-1-u \choose d-1} \right)}{d},
\end{equation*}
as desired.
\end{proof}

The next nearly immediate corollary gives a lower bound on the number of non-present edges when $u$ is on the order of $n$. 

\begin{corollary}\label{absent cd}
Let $D>1$ be a constant. There is some $C=C(D)>0$ such that if $H$ is a hypergraph on $[n']$, where $n-\ln n \leq n' \leq n$, with $u(H) \geq \frac{n}{D}$ that has a longest path of length $h$, but no weak cycle of length $h+1$, then there are at least $C n^d$ absent edges in $H$ such that the addition of any one of these non-present edges creates a weak cycle of length $h+1$.
\end{corollary}
\begin{proof}
The previous corollary gives the number of such absent beneficial edges is at least
\begin{align*}
\frac{ u(H) \left( {n'-1\choose d-1} - {n'-1-u(H) \choose d-1} \right)}{d} &\geq \frac{\frac{n}{D} \left( {n'-1 \choose d-1} - {n'-1-\frac{n}{D} \choose d-1} \right)}{d} \\& = \frac{ n^d}{D \, d!} \left( 1 - \left(1-\frac{1}{D} \right)^d +O\left( \frac{\ln n}{n}\right) \right) \geq C n^d,
\end{align*}
for any $C<\frac{1}{D \, d!} \left( 1 - \left(1-\frac{1}{D}\right)^d\right)$ and $n$ sufficiently large.
\end{proof}

\section{All non-expanding sets must be large}\label{sec: nonexpanding}

As noted before, to show that w.h.p. P\'{o}sa sets are large, we'll show that the smallest non-expanding set must likely be on the order of $n$. 
\begin{lemma}\label{u H bound}
Let $|\omega(n)| \leq \ln \ln \ln n$ and $p=(d-1)! \frac{\ln n+\omega}{n^{d-1}}$. Then w.h.p.
\begin{equation*}
u(H_d(n,p)) \geq n/3^d.
\end{equation*}
\end{lemma}

These next few sections deal with showing the absence of non-expanding sets according to their size. Specifically, we break up this range into small $\left([1, n^{1/4}]\right)$ and medium $\left([n^{1/4}, n/3^d]\right)$ sized sets. We begin by showing the likely absence of small non-expanding sets of non-isolated vertices.

\subsection{No Small Non-expanding Sets}

\begin{lemma}\label{no small}
Let $|\omega(n)| \leq \ln \ln \ln n$ and $p=(d-1)!\frac{\ln n +\omega}{n^{d-1}}$. Then w.h.p. there are no non-expanding sets of non-isolated vertices of size at most $n^{1/4}.$
\end{lemma}

Note that it suffices to prove that w.h.p. there are no {\it minimal} non-expanding sets of non-isolated vertices of size at most $n^{1/4}$ (minimal meaning no proper subset). Before we begin the proof, let's prove a consequence of a non-expanding set being minimal.

\begin{lemma}\label{minimal nonexpanding}
Suppose $H$ is a hypergraph with a non-expanding set, $A,$ that has no proper non-expanding subsets. The induced hypergraph on $T:=A \cup N(A)$ is connected.
\end{lemma}

\begin{proof}
We prove this lemma by contradiction. Suppose the induced hypergraph is not connected. Then we can decompose the induced hypergraph into 2 disjoint hypergraphs, $T=T_1 \cup T_2$, where $T_i$ are non-empty and there are no edges between the vertex sets $T_1$ and $T_2$ (using only edges contained within $T$). Let $A_i = A \cap T_i$. Note that $A_i \cup N(A_i) \subset T_i$ or else there is a edge from $T_1$ to $T_2$, contradicting their separation. Hence the vertex sets $A_1 \cup N(A_1)$ and $A_2 \cup N(A_2)$ are disjoint, and so
\begin{align*}
|A_1 \cup N(A_1)| + |A_2 \cup N(A_2)| &= |A_1 \cup A_2 \cup N(A_1) \cup N(A_2) |  \\&= |A| + |N(A)|  < 3 |A| = 3|A_1| + 3|A_2|.
\end{align*}
Consequently, either $|A_1 \cup N(A_1)| < 3|A_1|$ or $|A_2 \cup N(A_2)| < 3|A_2|$. This forces that $|N(A_i)| < 2|A_i|$ for at least one of $i=1,2$, which in turn, contradicts minimality of the non-expanding set $A$.
\end{proof}

Now we begin the proof of Lemma \ref{no small}.

\begin{proof}
If $A$ is a minimal non-expanding set of non-isolated vertices of size at most $n^{1/5},$ then by the previous lemma, the induced hypergraph on $T':=A \cup N(A)$ is connected and necessarily 
\begin{equation*}
d \leq |T'| < n^{1/4}+2n^{1/4}=3n^{1/4}.
\end{equation*}
Hence to prove the lemma, it ``almost" suffices to show that w.h.p. there are no such sets $T$ with: 
\begin{itemize}
\item $|T|=:t \in \left[\ell:=\max\{4, d\}, n_1:=3n^{1/4}\right],$
\item the induced hypergraph on $T$ is connected, and
\item at least $\lceil t/3 \rceil$ vertices of $T$ have no neighbors outside of $T$.
\end{itemize} 
We say almost, because in addition, we need to deal with the excluded case $d=3, t=3$. Namely, we've excluded the case where $d=3$ and there is a non-expanding set $A$ with $|A\cup N(A)|=3.$ This event corresponds to either a pair of adjacent degree 1 vertices or a triplet of adjacent degree 1 vertices; the probability of both of these events can be easily shown to tend to zero using a simple first moment argument, and so their proofs are omitted. To finish off the proof of the lemma, we show that $E_1:=E[\# \text{ of such sets }T]$ tends to zero.

By the union bound over all sets of cardinality $t$ along with all further subsets of cardinality $\lceil t/3 \rceil$,
\begin{equation}\label{eqn 432}
E_1 \leq \sum_{t=\ell}^{n_1} {n \choose t} {t \choose \lceil t/3 \rceil} P_t \leq \sum_{t = \ell}^{n_1} n^t \, P_t,
\end{equation}
where $P_t$ is the probability that the induced hypergraph on $[t]$ is connected and there are no edges containing a vertex of $[\lceil t/3 \rceil]$ and a vertex of $[n]\setminus [t]$. Notice that the set of potential edges contained entirely within $[t]$ is disjoint from the set of potential edges containing a vertex of $[\lceil t/3 \rceil]$ and a vertex of $[n] \setminus [t]$; consequently, these two events are independent! Hence
\begin{equation}\label{eqn: htp qn}
P_t= P\left( \text{induced hypergraph on }[t]\text{ is connected}\right) q^{N_t},
\end{equation}
where $N_t$ is the number of sets of cardinality $d$ that contain at least one element of $[\lceil t/3 \rceil]$ and at least one element of $[n] \setminus [t]$. In general, if $A$ and $B$ are disjoint subsets of $[n]$, then the number of cardinality $d$ subsets of $[n]$ that contain at least one element of $A$ and at least one element of $[n] \setminus (A \cup B)$ is 
\begin{equation}\label{eqn: 10.15.1}
{n \choose d}-{n-|A| \choose d} - {|A|+|B| \choose d}+{ |B| \choose d};
\end{equation}
in our particular case, we have that
\begin{equation*}
N_t = {n \choose d}-{n-\lceil t/3 \rceil \choose d}-{t \choose d}+{t-\lceil t/3 \rceil \choose d}.
\end{equation*}
Uniformly over the range of $t$, we have that
\begin{align*}
N_t &= {n \choose d} \left( 1 - \left( 1 - \frac{\lceil t/3 \rceil}{n} + O\left(\frac{n_1^2}{n^2}\right) \right)^d + O(n_1^d/n^d) \right) \\&= {n \choose d} \left( \frac{ d \, \lceil t/3 \rceil}{n} + O\left( \frac{n_1^2}{n^2}\right)\right). 
\end{align*}
Hence
\begin{equation*}
q^{N_t}   \leq  \exp \left( -p {n \choose d}\left(\frac{d}{n}\lceil t/3 \rceil +O\left(\frac{n_1^2}{n^2} \right)\right)  \right) = \exp \left( - \lceil t/3 \rceil p {n-1 \choose d-1} + o(1) \right).
\end{equation*}
So for sufficiently large $n$, 
\begin{equation}\label{eqn: 9.29.1}
q^{N_t} \leq 2 \exp \left( - \frac{t}{3} \frac{p \, n^{d-1}}{(d-1)!}\right).
\end{equation}

Now let's take on the probability that the induced hypergraph on $[t]$ is connected from \eqref{eqn: htp qn}. The induced hypergraph on $[t]$ is distributed as $H_d(t,p)$. For a hypergraph on $t$ vertices to be connected, there must be at least $\Big \lceil \frac{t-1}{d-1} \Big \rceil$ edges present. Hence
\begin{align*}
P(H_d(t,p) \text{ is connected})  \leq { {t \choose d} \choose \lceil \frac{t-1}{d-1} \rceil } p^{\lceil \frac{t-1}{d-1}\rceil}  \leq \left( \frac{e \frac{t^d}{d!}}{\lceil \frac{t-1}{d-1} \rceil} p \right)^{\lceil \frac{t-1}{d-1}\rceil}.
\end{align*}
Taking the ceilings off both terms above increases this expression. In fact, we get that
\begin{align}\label{eqn: 10.5.1}
P(H_d(t,p) \text{ is connected}) &\leq \bigg ( e\,  t^{d-1} \, \frac{p}{(d-1)!} \underbrace{\frac{t}{t-1} \, \frac{d-1}{d} }_{\leq 1} \bigg )^{\frac{t-1}{d-1}}\nonumber \\ &\leq \left( e \, t^{d-1} \, \frac{p}{(d-1)!} \right)^{\frac{t-1}{d-1}}.
\end{align}
Therefore, using this bound along with \eqref{eqn: 9.29.1}, \eqref{eqn: htp qn} becomes
\begin{equation*}
P_t \leq 2 \left( e \, t^{d-1} \, \frac{p}{(d-1)!} \right)^{\frac{t-1}{d-1}} e^{-  \frac{t}{3}  p \, {n-1 \choose d-1}  }.
\end{equation*}
Plugging this bound into \eqref{eqn 432} gives
\begin{equation*}
E_1 \leq 2 \sum_{t=\ell}^{n_1} n^t \left( e \, t^{d-1} \, \frac{p}{(d-1)!} \right)^{\frac{t-1}{d-1}} e^{-  \frac{t}{3}  p \, {n-1 \choose d-1}  } =: 2\sum_{t= \ell}^{n_1} c_t.
\end{equation*}
We'll show that the dominant term in the sum is the first term and note that
\begin{equation*}
c_{\ell} = O\left( n^{\ell} \left( \frac{\ln n}{n^{d-1}}\right)^{\frac{\ell-1}{d-1}} e^{-\frac{\ell}{3} \, p \, {n-1 \choose d-1}} \right)\ll \frac{(\ln n)^2}{n^{\frac{\ell}{3}-1}} \to 0.
\end{equation*}

To show that the dominant contribution of the sum of $c_t$ is $c_{\ell}$, it suffices to prove that the ratio of consecutive terms uniformly tends to zero. Now
\begin{equation*}
\frac{c_{t+1}}{c_t} = n \left( e \, \frac{ p}{(d-1)!} \right)^{\frac{1}{d-1}} \left(\frac{t+1}{t}\right)^t \, t \,  e^{-\frac{p}{3}{n-1 \choose d-1}}.
\end{equation*}
Therefore,
\begin{equation*}
\frac{c_{t+1}}{c_t} = O\left( n \, \frac{(\ln n)^{1/(d-1)}}{n} \, n_1 \left( \frac{\ln \ln n}{n} \right)^{1/3} \right) \ll \frac{\ln n}{n^{1/12}} \to 0;
\end{equation*}
whence 
\begin{equation*}
E_1 \leq 2 \sum_{t=\ell}^{n_1} c_t = O(c_{\ell}) \to 0,
\end{equation*}
as desired. Note that if the right endpoint value $n_1$ is much larger than $n^{1/4},$ then the ratio of consecutive terms of $c_t$ would start to become more than 1, and we could not simply bound the sum by just the order of the first summand.
\end{proof}

\subsection{No Medium Non-expanding Sets}

Now let's consider the number, $X$, of non-expanding sets of size in $[n^{1/4}, n/3^d]$. In particular, we relax the condition that the non-expanding set contains no isolated vertices anticipating that this restriction is unnecessary for larger sets of vertices. To finish off the proof of Lemma \ref{u H bound}, we want to show that w.h.p $X=0$, which we do by proving that $E[X] \to 0$. 

Note that 
\begin{equation*}
E[X] = \sum_{A, |A|=n^{1/4}}^{n/3^d} \sum_{B, |B|=0}^{2|A|-1} P\left( \text{the neighbor set of }A\text{ is }B\right),
\end{equation*}
and in particular,  
\begin{equation*}
P\left( \text{the neighbor set of }A\text{ is }B\right) = P(A, B) q^{{n \choose d}-{n-|A| \choose d}-{|A|+|B| \choose d}+{|B| \choose d}},
\end{equation*}
where $P(A,B)$ is the probability that each vertex in $B$ is in an edge, using only vertices in $A \cup B$, with at least one vertex of $A$; moreover, the vertices from $A$ are not in a edge with any vertex from $[n] \setminus (A \cup B)$, which gives rise to the $q$ term, see \eqref{eqn: 10.15.1} for the exponent. Also, we define $P(A,B)=1$ if $B = \emptyset$. By symmetry, $P(A,B)$ depends only upon the cardinalities of $A$ and $B$, which we denote by $P(a,b)$. Thus
\begin{equation}\label{eqn: 9.18.1}
E[X] = \sum_{a=n^{1/4}}^{n/3^d} \sum_{b=0}^{2a-1} {n \choose a}{n-a \choose b} P(a,b) q^{{n \choose d}-{n-a \choose d}-{a+b\choose d}+{b \choose d}}.
\end{equation}
The only non-explicit term in this expectation is $P(a,b)$. For sets $A$ with $a$ near the order of $n$, we would expect that the condition that each vertex of $B$ to be in at least one edge with a vertex of $B$ is not terribly restrictive. In fact, for $a \in [6n/\ln n, n/3^d]$, the trivial bound $P(a,b) \leq 1$ will be sufficient for our purposes here. However, for the remaining $a$, we will need to take on this $P(a,b)$ term. In fact, we establish a bound on $P(a,b)$ that is interesting in its own right; then, using this bound, we show that the sum in \eqref{eqn: 9.18.1} over remaining $a \in [n^{1/4}, 6n/\ln n]$ tends to zero.

\subsection{No non-expanding sets of size in $[6n/\ln n, n/3^d]$}

\begin{lemma}\label{no large}
Uniformly over $|\omega| \leq \ln \ln \ln n$ and $p = (d-1)! \frac{\ln n + \omega}{ n^{d-1} }$,
\begin{equation*}
E_1:=\sum_{a=6n/\ln n}^{n/3^d} \sum_{b=0}^{2a} {n \choose a}{n-a \choose b} P(a,b) q^{{n \choose d}-{n-a \choose d}-{a+b\choose d}+{b \choose d}} \to 0.
\end{equation*}
\end{lemma}

\begin{proof}
For ease of notation, let $\nu_0:=6n/ \ln n$, $D=3^d$ and $\nu_1:=n/D.$ Using the trivial bound $P(a,b) \leq 1$, we have that  
\begin{equation*}
E_1 \leq \sum_{a=\nu_0}^{\nu_1} \sum_{b=0}^{2a} {n \choose a}{n \choose b} q^{ {n \choose d}-{n-a\choose d} - {a+b \choose d} + {b \choose d}} =: \sum_{a,b} \gamma_b^a.
\end{equation*}
We will first prove that uniformly over $a$, the $b=2a$ term, $\gamma_{2a}^a,$ dominates the other $b$ terms, by showing that the ratio of consecutive terms is at least 3. Then, by a more careful analysis, we prove that the sum of $\gamma_{2a}^a$ tend to zero.

Uniformly over $a \in [\nu_0, \nu_1],$ with $n$ sufficiently large,
\begin{equation*}
\frac{\gamma_{b+1}^a}{\gamma_b^a} = \frac{n-b}{b+1} q^{ -{a+b \choose d-1} + {b \choose d-1} } \geq \frac{n- 2\frac{n}{D} }{2\frac{n}{D} +1} \cdot 1 \geq  3;
\end{equation*}
whence
\begin{equation*}
E_1 = \sum_{a=\nu_0}^{\nu_1} \sum_{b=0}^{2a} \gamma_b^a \leq \sum_{a=\nu_0}^{\nu_1} \gamma_{2a}^a  \sum_{b=0}^{2a} \left(\frac{1}{3}\right)^{2a-b} \leq \frac{3}{2} \sum_{a=\nu_0}^{\nu_1} \gamma_{2a}^a.
\end{equation*}
Now using the inequalities $q \leq e^{-p}$, ${n \choose a} \leq \left( \frac{e \, n }{a} \right)^a$ and ${n \choose 2a} \leq \left( \frac{e \, n }{2a} \right)^{2a}$, after some simplification we get
\begin{equation}\label{eqn 463}
E_1 \leq  \frac{3}{2} \sum_{a=\nu_0}^{\nu_1} \exp \left( a \ln\left( \frac{e^3 n^3}{4 a^3} \right)  - p \left( {n \choose d}-{n-a \choose d} - {3a \choose d} + {2a \choose d} \right) \right).
\end{equation}
Using the fact that $(m)_d =m(m-1)\dots (m-d+1)= m^d - {d \choose 2}m^{d-1}+O(m^{d-2})$, the dominant terms in the first two binomial terms in \eqref{eqn 463} can be found by
\begin{align*}
d! \left( {n \choose d}-{n-a \choose d} \right) &= (n)_d - (n-a)_d \\ &= n^d - (n-a)^d - {d \choose 2} \left( n^{d-1} - (n-a)^{d-1} \right) +O(n^{d-2}) \\ &= -\sum_{i=1}^d {d \choose i} (-a)^{i}n^{d-i} + {d \choose 2} \sum_{i=1}^{d-1}{d-1 \choose i} (-a)^{i} n^{d-1-i} + O(n^{d-2}).
\end{align*}
Using a similar approximation for the remaining two binomial expressions as well as using the fact that the error in the exponent $p n^{d-2} \to 0$, \eqref{eqn 463} becomes 
\begin{equation}\label{eqn 464}
E_1 \leq 2 \sum_{a=\nu_0}^{\nu_1} \exp \left( a \cdot g_n(a)\right),
\end{equation}
where 
\begin{align*}
g_n(a) &:= \ln \left( \frac{e^3 n^3}{4 a^3}  \right) - \frac{p}{d!} \left(   \sum_{i=1}^{d} {d \choose i} (-a)^{i-1} n^{d-i}\right) + \frac{p}{d!} \left( {d \choose 2} \sum_{i=1}^{d-1} {d-1 \choose i} (-a)^{i-1} n^{d-1-i} \right) \\& -\frac{p}{d!} \left(  \left(2^d - 3^d\right) a^{d-1} + { d\choose 2} \left(3^{d-1} - 2^{d-1} \right) a^{d-2} \right). 
\end{align*}
We will  show that this function $g_n(a)$ is increasing, for $a \in [\nu_0, \nu_1]$, by computing $g_n'(a)$. But first, taking this fact for granted, let's show that $E_1 \to 0$. 

Note that for any $\alpha \in (0,1)$ fixed (eventually $\alpha = 1/3^d$), we have that
\begin{equation*}
g_n(\alpha n) = - \frac{p}{d!} \left( \sum_{i=1}^d {d \choose i} (-\alpha)^{i-1} n^{d-1} \right) + \frac{p}{d!} (3^d-2^d) \left( \alpha n\right)^d + O(1).
\end{equation*}
In particular,
\begin{equation*}
g_n(\alpha n) \leq \frac{p n^{d-1}}{d!} \left( - d + \sum_{i=2}^d {d \choose i} (\alpha)^{i-1} + \alpha^{d-1} 3^d \right) + O(1);
\end{equation*}
whence
\begin{equation*}
g_n(\alpha n) \leq  \frac{p n^{d-1}}{d!} \left( - d + \alpha \, 2^d + \alpha \, 3^d \right) + O(1);
\end{equation*}
By choosing $\alpha = 1/3^d$, we have that
\begin{equation*}
g_n(n/3^d) \leq  \frac{p n^{d-1}}{d!} \left( - d + 5/3 \right) + O(1),
\end{equation*}
and uniformly over $p$, 
\begin{equation*}
g_n(n/3^d) \leq - (\ln n)/d.
\end{equation*}
Therefore
\begin{equation*}
E_1 \leq 2 \sum_{a=\nu_0}^{\nu_1} \exp \left( -a(\ln n)/d\right) = 2 \sum_{a=\nu_0}^{\nu_1} \frac{1}{n^{a/d}} \to 0,
\end{equation*}
as desired. 

All that remains of the proof of this lemma is to show that $g_n$ is indeed increasing on $[\nu_0, \nu_1]$, which we do by proving that $g_n'(a) \geq (\ln n)/(3n)$ for all $a$ in this range. Uniformly over $a$ and $p$, we have that
\begin{equation*}
g_n'(a) = \frac{-3}{a} +\frac{p}{d!} \left(  \sum_{i=2}^{d} {d \choose i} (i-1) (-a)^{i-2} n^{d-i} \right)  -\frac{p}{d!} \left(  \left(2^d - 3^d\right) (d-1) a^{d-2}  \right) + o(\ln n/n).
\end{equation*}
Neglecting some of the positive terms as well as bounding $a$ between $\nu_0$ and $\nu_1$, we find that
\begin{equation*}
g_n'(a) \geq \frac{-3}{\nu_0} + \frac{p}{d!} \left(  {d \choose 2}n^{d-2} -  \sum_{i=3}^d {d \choose i} (i-1) (\nu_1)^{i-2}n^{d-i} \right) + o(\ln n/n).
\end{equation*}
To deal with this sum, note that 
\begin{equation*}
\sum_{i=3}^d {d \choose i} (i-1) (\nu_1)^{i-2} n^{d-i} \leq \nu_1 \, n^{d-3} \sum_{i=3}^d {d \choose i} i \leq \frac{n^{d-2}}{3^d} \, d \, 2^{d-1}.
\end{equation*}
Using this inequality and the fact that $p n^{d-1}/(d-1)! = \ln n + o(\ln n)$, we find that 
\begin{equation*}
g_n'(a) \geq \frac{-3}{6n/(\ln n)} + \frac{\ln n}{d \, n} \left( {d \choose 2} - \frac{d \, 2^{d-1}}{3^d} \right) + o(\ln n/n);
\end{equation*}
whence
\begin{equation*}
g_n'(a) \geq \frac{\ln n}{2n} \left( -1+ (d-1) - \left(2/3\right)^d \right) + o(\ln n/n) \geq \frac{\ln n}{3n},
\end{equation*}
since $d\geq 3$, which concludes the proof that $E_1 \to 0$. 
\end{proof}

\subsection{Bounding $P(a,b)$}

\begin{lemma}
Let $A$, $a:=|A|$, and $B$, $b:=|B|$, be disjoint sets of vertices, where $a$, $b \geq 1$ and $a+b \geq d$. Suppose each of the ${a + b \choose d}$ potential edges of cardinality $d$ is present with probability $p\in(0,1)$ independently of one another. Then the probability each vertex of $B$ is adjacent to $A$, denoted $P(a,b)$, is bounded above by
\begin{equation*}
P(a,b) \leq \left( 1 - q^{ {a + b - 1 \choose d-1} - {b-1 \choose d-1} } \right)^{ \lceil \frac{b}{d-1} \rceil}.
\end{equation*}
\end{lemma}

\begin{proof}
We introduce a greedy edge finding algorithm which finds at least $\lceil b/(d-1) \rceil$ edges. Analyzing this algorithm will deliver our bound on $P(a,b)$. The algorithm is as follows. Let $w_1, \ldots, w_b$ be the vertices of $B$.
\begin{enumerate}
\item We begin at vertex $v_1:=w_1$. Check the potential edges, one after another, containing $w_1$ and at least one vertex of $A$. We stop checking the moment that we find the first present edge, which we denote $e_1$, and go to step 2.
\item Let $v_2$ be the lowest index vertex of $B \setminus e_1$. Check the potential edges, one after another, containing $v_2$ and at least one vertex of $A$ (without looking at any previously checked present or non-present edges). We stop checking the moment that we find the first present edge, which we denote $e_2$, and go to step 3.
\item Let $v_3$ be the lowest index vertex of $B \setminus (e_1 \cup e_2)$. Again we check the potential edges containing $v_3$ and at least one vertex of $A$ one after another. And so on until $B \setminus (e_1 \cup e_2 \cup \ldots)$ is empty.
\end{enumerate}

If each vertex of $B$ is adjacent to at least one vertex of $A$, then in the $k^{th}$ step, as long as $B\setminus (e_1 \cup \ldots \cup e_{k-1})$ is non-empty, we must find at least one edge containing $v_k$. Furthermore, each new found edge $e_i$ can contain at most $(d-1)$ vertices of $B$, since at least one of $e_i$'s vertices is in $A$. So on the event corresponding to $P(a,b)$, we must find at least $\lceil b/(d-1) \rceil$ edges; in other words, 
\begin{equation*}
P(a,b) \leq P(\text{we find }e_1, e_2, \ldots, e_{\lceil b/(d-1) \rceil} \text{ in the greedy algorithm}).
\end{equation*}
For ease in bounding the probability on the right, let $P_1:=P(\text{we find }e_1)$ and for $j\geq 2$,
\begin{equation*}
P_j := P(\text{find }e_j | \text{we find }e_1, \ldots, e_{j-1}).
\end{equation*}
Since we need to find $e_1, \ldots, e_{j-1}$ in order to find $e_j$, the bound on $P(a,b)$ above becomes
\begin{equation*}
P(a,b) \leq P_1 \cdot P_2 \cdots P_{\lceil b/(d-1) \rceil}.
\end{equation*}
To finish the proof of the lemma, it is sufficient to show that for each $j$,
\begin{equation}\label{eqn: 9.26.3}
P_j \leq 1- q^{{a+b-1 \choose d-1} - {b-1 \choose d-1}}.
\end{equation}
For $j=1$, we find $e_1$ iff at least one of the ${a+b-1 \choose d-1}-{b-1 \choose d-1}$ potential edges containing $v_1$ is present; so $P_1$ actually attains equality in \eqref{eqn: 9.26.3}. 

Now let's calculate $P_j$ for $j \geq 2$. Let $M_j$ be the (random) number of previously checked edges containing $v_j$ that we know are not present before the $j^{th}$ step. So at the beginning of the $j^{th}$ step of the process, we have ${a+b-1 \choose d-1}-{b-1 \choose d-1}-M_j$ unchecked edges containing $v_j$ and at least one vertex of $A$. So
\begin{align*}
P_j &= \sum_{m=0}^{{a+b-1 \choose d-1}-{b-1 \choose d-1}} P(\text{find }e_j, M_j = m | \text{ we find }e_1, \ldots, e_{j-1}) \\&= \sum_m P(\text{find }e_j | M_j =m) \cdot P(M_j =m | \text{ we find }e_1, \ldots, e_{j-1}),
\end{align*}
and in particular,
\begin{equation*}
P(\text{find }e_j | M_j=m) = 1 - q^{{a+b-1 \choose d-1}-{b-1 \choose d-1}-m} \leq 1-q^{{a+b-1 \choose d-1}}.
\end{equation*}
Using this inequality above yields the bound \eqref{eqn: 9.26.3}.
\end{proof}

To bound $P(a,b)$ for the range of $p$, $a$ and $b$ that we care about, we'll use the following bound.

\begin{corollary}\label{pab bound}
If $a \geq d, b \in [1,2a]$ such that $2a \, p^{\frac{1}{d-1}} \leq 1$, then
\begin{equation*}
P(a,b) \leq \left( 2 a \, p^{\frac{1}{d-1}} \right)^b. 
\end{equation*}
\end{corollary}

\begin{proof}
Using Bernoulli's Inequality and then the fact that ${a+b-1\choose d-1} - {b-1 \choose d-1} \leq a^{d-1} 2^{d-1}$, we have that 
\begin{align*}
P(a,b) &\leq \left( 1 - (1-p)^{{a+b-1 \choose d-1}-{b-1 \choose d-1}}\right)^{\lceil \frac{b}{d-1} \rceil} \\
& \leq \left( \left( {a+b-1\choose d-1} - {b-1 \choose d-1} \right) p \right)^{\lceil \frac{b}{d-1} \rceil} \leq \left( a^{d-1} 2^{d-1} p \right)^{\lceil \frac{b}{d-1} \rceil}.
\end{align*}
By hypothesis, $a^{d-1} 2^{d-1} p \leq 1$, so we can take off the ceilings in the above statement to obtain the desired inequality.
\end{proof}

\subsection{No non-expanding sets of size in $[n^{1/4}, 6n/\ln n]$}

\begin{lemma}\label{no medium}
Uniformly over $|\omega| \leq \ln \ln \ln n$ and $p=\frac{(d-1)! (\ln n + \omega)}{n^{d-1}}$,
\begin{equation}\label{eqn 451}
E_2 :=\sum_{a=n^{1/4}}^{6n/\ln n} \sum_{b=0}^{2a} {n \choose a}{n-a \choose b} P(a,b) q^{{n \choose d}-{n-a \choose d}-{a+b \choose d}+{b \choose d}} \to 0.
\end{equation}
\end{lemma}

\begin{proof}
For ease of notation, let $\bar{\nu}_0 = n^{1/4}$ and $\bar{\nu}_1 = 6n/ \ln n$. For $a \in [\bar{\nu}_0, \bar{\nu}_1]$, we have that $2 a \, p^{1/(d-1)} <1$, so by applying our bound for $P(a,b)$ from Corollary \ref{pab bound} to \eqref{eqn 451}, we obtain that
\begin{equation*}
E_2 \leq \sum_{a=\bar{\nu}_0}^{\bar{\nu}_1}  \sum_{b=0}^{b=2a} {n \choose a} { n-a \choose b} \left(2 a p^{\frac{1}{d-1}} \right)^b q^{ { n \choose d} - {n-a \choose d} - {a+b \choose d} + {b \choose d} }.
\end{equation*}
Now using the inequalities ${k \choose \ell} \leq (e k / \ell)^{\ell}$ on the first two binomial terms as well as $q \leq e^{-p}$, we find that
\begin{equation*}
E_2 \leq  \sum_{a=\bar{\nu}_0}^{\bar{\nu}_1}  \sum_{b=0}^{b=2a} \left( \frac{e n}{a} \right)^a \left( \frac{ e n }{b} \right)^b  \left(2 a p^{\frac{1}{d-1}} \right)^b e^{-p\left( { n \choose d} - {n-a \choose d} - {a+b \choose d} + {b \choose d} \right)} =: \sum_{a,b} \gamma_b^a.
\end{equation*}
We will bound this sum of $\gamma_a^b$ in two steps. As before, the term $\gamma_{2a}^a$ will be shown to dominate the other $\gamma_b^a$; then we will prove that $\gamma_{2a}^a \leq 3/n^{a/5}$. 

Computing the ratio of consecutive terms, we find that
\begin{align*}
\frac{\gamma_{b+1}^a}{\gamma_b^a} &= \frac{e n}{b+1} \left( \frac{b}{b+1} \right)^b \left( 2 a p^{1/(d-1)} \right) e^{p\left( {a+b \choose d-1}-{b \choose d-1}\right)} \geq \frac{n}{2a} \left( 2 a  p^{1/(d-1)} \right) \\ &\geq (\ln n)^{1/(d-1)}.
\end{align*}
Therefore, for each $a \in [\bar{\nu}_0, \bar{\nu}_1]$, we have that
\begin{equation*}
\sum_{b=0}^{2a} \gamma_b^a \leq \gamma_{2a}^a \sum_{k=0} \left( \ln n \right)^{-k/(d-1)} \leq 2 \gamma_{2a}^a,
\end{equation*}
for $n$ sufficiently large. Therefore
\begin{equation*}
E_2 \leq 2\sum_{a=\bar{\nu}_0}^{\bar{\nu}_1} \left( \frac{e n }{a} \right)^a \left( \frac{e n }{2a} \right)^{2a}  \left(2 a p^{\frac{1}{d-1}} \right)^{2a} e^{-p \left( { n \choose d} - {n-a \choose d} - {3a \choose d} + {2a \choose d}\right) }, 
\end{equation*}
or equivalently
\begin{equation}\label{eqn 9.26.1}
E_2 \leq 2 \sum_{a=\bar{\nu}_0}^{\bar{\nu}_1}\exp \left( a \ln \left( \frac{e^3 n^3}{a} p^{\frac{2}{d-1}} \right) - p\left( { n \choose d} - {n-a \choose d} - {3a \choose d} + {2a \choose d}\right) \right). 
\end{equation}
If we approximate the non-negligible terms in the binomial expressions in \eqref{eqn 9.26.1} just as we did for \eqref{eqn 464}, we obtain that 
\begin{equation}\label{eqn 9.26.2}
E_2 \leq 3 \sum_{a=\bar{\nu}_0}^{\bar{\nu}_1} \exp \left( a \, f_n(a) \right),
\end{equation}
where
\begin{align*}
f_n(a) &:= \ln \left( \frac{e^3 n^3 p^{\frac{2}{d-1}}}{a}  \right) - \frac{p}{d!} \left(  \sum_{i=1}^{d} {d \choose i} (-a)^{i-1} n^{d-i} - {d \choose 2} \sum_{i=1}^{d-1} {d-1 \choose i} (-a)^{i-1} n^{d-1-i} \right) \\& -\frac{p}{d!} \left(  \left(2^d - 3^d\right) a^{d-1} + { d\choose 2} \left(3^{d-1} - 2^{d-1} \right) a^{d-2} \right). 
\end{align*}

To bound $E_2$, we will show that $f_n$ is convex on $[\bar{\nu}_0, \bar{\nu}_1].$ But first, let's take this fact for granted right now and show that $E_2 \to 0$. Convexity of $f_n$ implies that
\begin{equation*}
\max_{a \in [\bar{\nu}_0, \bar{\nu}_1]} f_n(a) \leq \max \{f_n(\bar{\nu}_0), f_n(\bar{\nu}_1)\}.
\end{equation*}
Computing these last two values, we find that
\begin{align*}
f_n(\bar{\nu}_0) &= -\frac{1}{4}\ln n + O(\ln \ln n) \leq -\frac{1}{5}\ln n, \\ 
f_n(\bar{\nu}_1) &= - \ln n + O(\ln \ln n) \leq -\frac{1}{5}\ln n.
\end{align*}
Hence \eqref{eqn 9.26.2} becomes
\begin{equation*}
E_2 \leq 3 \sum_{a=\bar{\nu}_0}^{\bar{\nu}_1} \exp \left( a \, \max_a f_n(a) \right) \leq 3\sum_{a=\bar{\nu}_0}^{\bar{\nu}_1} n^{-a/5} \to 0.
\end{equation*}

All that remains is to show that $f_n$ is actually convex. This is done by computing $f_n''$ and showing that $f_n'' > 0$ uniformly. Note that
\begin{align*}
f_n''(a) &= \frac{1}{a^2} - \frac{p}{d!} \left(   \sum_{i=3}^{d} {d \choose i} (i-1)(i-2) (-a)^{i-3} n^{d-i} \right) \\&+  \frac{p}{d!}\left( {d \choose 2} \sum_{i=3}^{d-1} {d-1 \choose i} (i-1)(i-2)(-a)^{i-3} n^{d-1-i} \right) \\& -\frac{p}{d!} (d-2) \left(  \left(2^d - 3^d\right) (d-1) a^{d-3} + { d\choose 2} \left(3^{d-1} - 2^{d-1} \right) (d-3) a^{d-4} \right), 
\end{align*}
where the second sum and the last term are understood to be zero when $d=3$. Uniformly over $a \in [\bar{\nu}_0, \bar{\nu}_1]$ all the terms other than the first are negligible and we have that
\begin{equation*}
f_n''(a) \geq \frac{1}{\bar{\nu}_1^2} + O\left( \frac{\ln n}{n^2}\right) > \frac{1}{37} \frac{(\ln n)^2}{n^2},
\end{equation*}
for sufficiently large $n$. Note that if the right endpoint, $\nu_1$, was larger and closer to the order of $n$, this leading order term in the above would be washed out by the error.
\end{proof}

\section{Completing the proof of the main result}\label{sec: increment}

\begin{theorem}
Let $c_n \to c \in \mathbb{R}$ and $p=(d-1)! \frac{\ln n + c_n}{n^{d-1}}$. W.h.p. $H_d(n,p)$ has a weak cycle that spans the non-isolated vertices. 
\end{theorem}

We will prove this theorem through a couple of lemmas, which we outline here. We start with some edge probability $p_0$ relatively far below $p$ and incrementally increase our edge probability. After each increment, we will w.h.p. be assured of completing a longest path into a weak cycle; then, we break apart this weak cycle to increase the length of a longest path. After sufficiently many steps, increasing our edge probability to $p_1<p$, we will have w.h.p. a weak cycle in $H_d(n,p_1)$ spanning the non-isolated vertices of $H_d(n,p_0)$. Unfortunately, vertices that are isolated in $H_d(n,p_0)$ may possibly be non-isolated in $H_d(n,p_1)$. In fact, since we run so many steps from $p_0$ to $p_1$, there is a positive limiting probability that this happens. However, we can show that not many vertices become non-isolated in this way; in particular, w.h.p. $H_d(n,p_1)$ has a weak cycle on all but $\ln n$ of its non-isolated vertices. Next, we begin to incrementally increase the edge probability again, this time from $p_1$ to $p$, but now we run much fewer steps. Again, we show that $H_d(n,p)$ has a weak cycle on the non-isolated vertices of $H_d(n,p_1)$, but since $p_1$ is near enough to $p$, w.h.p. the isolated vertices of $H_d(n,p_1)$ stay isolated in $H_d(n,p)$. This completes the proof of the theorem.

Truth be told, we increment our hyperedge probability in two steps, unlike Bollob\'{a}s' proof \cite{BB} in the graph case which increments in one step, because we wish to determine the probability of weak Hamiltonicity in the critical window (where limiting probability is strictly between $0$ and $1$). If we were concerned only with showing that w.h.p $H_d(n, p=(d-1)! \frac{\ln n +\omega}{n^{d-1}},$ $\omega \to \infty$ is weak Hamiltonian, then we could increment in just one step as well. 

We begin with showing that for $p_1$ near but below $p$, the isolated vertices in $H_d(n, p_1)$ stay isolated in $H_d(n,p)$. Let $V_0(H)$ denote the isolated vertices of the hypergraph $H$.

\begin{lemma}
Let $c_n \to c \in \mathbb{R}$ and $p=(d-1)! \frac{\ln n + c_n}{n^{d-1}}$.  Suppose $p_1 := p - \frac{(\ln n)^3}{n^d}$. Then, w.h.p. $V_0(H_d(n, p_1))=V_0(H_d(n,p))$. In other words, the isolated vertices of $H_d(n,p_1)$ are still isolated in $H_d(n,p)$ (under the usual containment). 
\end{lemma}

\begin{proof}
We can construct $H_d(n,p)$ from independent random hypergraphs $H_d(n,p_1)$ and $H_d(n,p^*)$ with $p^*=\frac{p-p_1}{1-p_1}$, where an edge is present in $H_d(n,p)$ if and only if this edge is present in at least one of $H_d(n,p_1)$ and $H_d(n,p^*)$. This is because an edge is present in union of $H_d(n,p_1)$ and $H_d(n,p^*)$ with probability
\begin{equation*}
P(e \in H_d(n,p_1))+P(e \in H_d(n,p^*)-P(e \in H_d(n,p_1) \cap H_d(n,p^*)) = p_1+p^*-p_1 \, p^* = p.
\end{equation*}
We will find that there are not many isolated vertices of $H_d(n,p_1)$; then, we will show that it is unlikely that any of these few vertices are {\it not} isolated in $H_d(n, p^*)$. In this case, $V_0(H_d(n,p_1))=V_0(H_d(n,p))$. 

By Lemma \ref{lemma: converge in distribution}, the number of isolated vertices in $H_d(n, p_1)$ is asymptotically Poisson with mean $e^{-c}$. Hence, for any $\omega \to \infty$, w.h.p the number of isolated vertices is less than $\omega$. For instance, w.h.p. $|V_0(H_d(n,p_1))| \leq \ln n$. Further,
\begin{equation*}
p^*=\frac{p-p_1}{1-p_1} = \frac{1}{1-p_1} \frac{(\ln n)^3}{n^d} \leq  \frac{2 \,(\ln n)^3}{n^d},
\end{equation*}
for $n$ large enough. 

To prove the lemma, it suffices to show that the probability that $|V_0(H_d(n,p_1))| \leq \ln n$ and that some vertex of $V_0(H_d(n,p_1))$ is not isolated in $H_d(n,p^*)$ tends to zero.  If we break up this event across the realization of $V_0(\circ)$, this probability becomes 
\begin{equation*}
\sum_{A \subset [n], |A| \leq \ln n} P(\{V_0(H_d(n,p_1))=A \} \cap \{ A \setminus V_0(H_d(n,p^*)) \neq \emptyset \} ).
\end{equation*}
Independence of the two hypergraphs above, allows us to break up this probability as
\begin{equation}\label{eqn: 10.1.1}
\sum_{A \subset [n], |A| \leq \ln n} P(V_0(H_d(n,p_1))=A  ) \cdot P(  A \setminus V_0(H_d(n,p^*)) \neq \emptyset  ).
\end{equation}
Now we take on this latter probability. For any set of vertices $A$ with $|A| \leq \ln n$, by the union bound and symmetry, we have that
\begin{align*}
P(A \setminus V_0(H_d(n,p^*)) \neq \emptyset  ) &\leq |A| \, P(\text{generic vertex is not isolated in }H_d(n,p^*)) \\ &\leq (\ln n) \, {n-1 \choose d-1} p^* = O\left(\frac{(\ln n)^4}{n}\right).
\end{align*}
Using this bound in \eqref{eqn: 10.1.1} and ``summing" over all such sets $A$ shows that our desired probability tends to zero. 
\end{proof}

Now we begin incrementally increasing the edge probability. As before, $p_1 = p-\frac{(\ln n)^3}{n^d}$. Let $p_0= p_1 - \frac{2^{d+4}}{C} \frac{\ln n}{n^d}$, where $C=C(3^d)$ is defined in Corollary \ref{absent cd}. At each step, we increase our edge probability by roughly $\Delta p := \frac{2}{C} \frac{\ln n}{n^d}$, and we will do $k_0 = \lceil 2^{d+3} \frac{n}{\ln n} \rceil$ steps in the first run and $k_1= \lceil \ln n \rceil$ steps in the second. In particular, note that 
\begin{equation*}
p_0 < p_0+\Delta p < p_0 + 2 \Delta p < \ldots < p_0 + k_0 \Delta p < p_1
\end{equation*}
and
\begin{equation*}
p_1 < p_1+\Delta p < p_1 + 2 \Delta p < \ldots < p_1+k_1 \Delta p <p.
\end{equation*}

\begin{lemma}
$\mathbf{(i)}.$ W.h.p. there is a weak cycle in $H_d(n, p_1)$ spanning the non-isolated vertices of $H_d(n,p_0)$.

$\mathbf{(ii)}.$ W.h.p. there is a weak cycle in $H_d(n, p)$ spanning the non-isolated vertices of $H_d(n,p_1)$. 
\end{lemma}

\begin{proof}
\textbf{(i).} Let $H^0$ be distributed as $H_d(n, p_0)$ and for each $i \in \{1, 2, \ldots, k_0\}$, let $H^i$ be independent copies of $H_d(n, \Delta p)$. Now, let $H(j)$ be the random hypergraph where an edge is in $H(j)$ if and only if this edge is present in at least one of $H^0, H^1, \ldots, H^j$. By construction, we have that 
\begin{equation*}
H(0) \subset H(1) \subset H(2) \subset \ldots \subset H(k_0).
\end{equation*}
The key is that in going from $H(j)$ to $H(j+1)$, we add potential edges to $H(j)$ by looking at an independent copy of $H_d(n, \Delta p)$. Further, note that $H(k_0)$ is distributed as $H_d(n, p')$, where 
\begin{equation*}
p' = 1-(1-p_0) (1-\Delta p)^{k_0} \leq p_0+k_0 \, \Delta p \leq p_1.
\end{equation*}
Our goal here is to show that $H(k_0)$ has a weak cycle spanning the isolated vertices of $H(0)$. We will do this by showing that in moving from $H(j)$ to $H(j+1)$, it is likely that \textbf{either} we extend a longest path of $H(j)$ \textbf{or} $H(j)$ actually already has a weak cycle spanning the isolated vertices of $H(0)$. 

Before, we analyze the increment steps, let's consider some likely events in $H(0) = H_d(n,p_0)$. From the hypergraphic version of de la Vega's Theorem (Lemma \ref{hyper de la vega}), w.h.p., there is a path of length at least $n(1-\tfrac{2^{d+2}}{\ln n-\omega})$. For simplicity, note that this path is longer than $n-k_0,$ which is the lower bound that we'll use instead. In addition, as we noticed in the proof of the previous lemma, the number of vertices of degree zero is at most $ \ln n$ (see Lemma \ref{lemma: converge in distribution}). By Lemma \ref{u H bound}, w.h.p. $u(H_d(n,p_0)) \geq n/3^d$. Further, the non-isolated vertices of $H_d(n,p_0)$ form a component (we prove this fact in Lemma \ref{lemma: 10.6.1} in the appendix). Let $\mathcal{E}$ be the intersection of these ``w.h.p." events in $H(0)$. 

We introduce a couple of definitions. For a hypergraph $H$, let $V_1(H)$ denote the set of non-isolated vertices of $H$. Also, for a set of vertices $W$, let $(H)_W$ denote the subgraph of $H$ induced on $W$. We will prove that w.h.p. $(H(k))_{V_1(H(0))}$ is weak Hamiltonian. On $\mathcal{E}$, we have that $u(H(0)) \geq n/D$, and $H(0)_{V_1(H(0))}$ is connected; in this case, we have that $u( (H(j))_{V_1(H(0))} ) \geq n/D$ as well (see $u$'s definition \eqref{def of uH}).

Now let's consider the event that $V_1(H(0))=A$ and $\mathcal{E}$ occurs; so $|A| \geq n - \ln n$, since there are at most $\ln n$ isolated vertices. Define $\ell_j$ as the length of the longest path in $(H(j))_A$ if this induced hypergraph is not weak Hamiltonian and define $\ell_j=|A|$ if it is. Clearly, $\ell_0 \leq \ell_1 \leq \ldots \leq \ell_{k_0}$. At each increment step, we want to either already have a weak Hamiltonian cycle ($\ell_j=\ell_{j+1}=|A|$) or we want a longest path to be extended ($\ell_j < \ell_{j+1}$).  In the case that this latter two events don't happen, necessarily $C n^d$ (specified) edges of $H^{j+1}$ must not be present; let's see why.

\begin{enumerate}
\item Suppose $\ell_j=\ell_{j+1} < |A|-1$ (again, we are on $\{V_1(H(0))=A\} \cap \mathcal{E}$). Since $(H(0))_A$ is connected (by $\mathcal{E}$), so is $(H(j))_A$. Now if $(H(j))_A$ has a weak cycle of length of $\ell_j+1\leq |A|-1$, then since $(H(j))_A$ is connected, there is an adjacent vertex to this cycle; so, we can break apart this weak cycle and extend a supposedly longest path, contradicting its maximum length. Therefore $(H(j))_A$ can not have a weak cycle of length $\ell_j+1$. Likewise, there can not be a cycle of length $\ell_{j}+1$ in $(H(j+1))_A$. By the corollary of P\'{o}sa's Lemma (Corollary \ref{absent cd}), there are at least $C n^d$ absent edges of $(H(j))_A$ whose addition would create a weak cycle of length $\ell_j+1$. Therefore, these $C n^d$ must also be absent in $(H(j+1))_A$.
\item Now consider the event where $\ell_j=\ell_{j+1}=|A|-1$. So we have a weak Hamiltonian path in $(H(j))_A$ and $(H(j+1))_A$, but no weak Hamiltonian cycle in either. Again by Corollary \ref{absent cd}, there are $C n^d$ absent edges of $H(j)$ that must stay absent in $H(j+1)$. 
\end{enumerate}

Therefore, on the event that $\ell_j = \ell_{j+1} < |A|$, there are $C n^d$ missing edges of $H(j)$ that are still absent in $H(j+1)$; in other words, there are $C n^d$ known missing edges of $H^{j+1}$. By independence of $H^{j+1}$ from $H(j)$, we have that
\begin{equation*}
P(\{V_1(H(0))=A\} \cap \mathcal{E} \cap \{\ell_j=\ell_{j+1} < |A|\}) \leq P(\{V_1(H(0))=A\} \cap \mathcal{E}) \cdot (1-\Delta p)^{C n^d}.
\end{equation*}
Now let's take on this last factor.
\begin{equation*}
(1-\Delta p)^{C n^d} \leq \exp \left( - \Delta p \, C\, n^d \right) = 1/n^2.
\end{equation*}
Hence
\begin{equation}\label{eqn breakdown}
P(\{V_1(H(0))=A\} \cap \mathcal{E} \cap \{\ell_j=\ell_{j+1} < |A|\}) \leq \frac{P(\{V_1(H(0))=A\} \cap \mathcal{E})}{n^2}. 
\end{equation}

By the hypergraph version of de la Vega's Theorem, on $\mathcal{E}$, there is a path in $H(0)$ of length $n - k_0$. Since we run $k_0$ steps, on $\mathcal{E}$, if we do not have a weak cycle of length $|A|$ in $H(k_0),$ then necessarily at some intermediate step, $\ell_j$ is less than $|A|$ and does not increase;  formally,
\begin{equation*}
\{\ell_k < |A|, \mathcal{E} \} \subset \mathop{\cup}_{j=0}^{k-1} \{\ell_j = \ell_{j+1} < |A|, \mathcal{E}\}.
\end{equation*}

Therefore, using the union bound as well as \eqref{eqn breakdown}, we get that
\begin{align}\label{eqn: 10.1.2}
P(\{V_1(H(0))=A\} \cap \mathcal{E} \cap \{\ell_k < |A| \}) &\leq \sum_{j=0}^{k_0-1} P(\{V_1(H(0))=A\} \cap \mathcal{E} \cap \{\ell_j=\ell_{j+1}<|A|\})\nonumber \\& \leq \frac{k_0}{n^2} \, P(\{V_1(H(0))=A\} \cap \mathcal{E}).
\end{align}
By ``summing" \eqref{eqn: 10.1.2} over all such vertex sets $A$, we get that
\begin{equation*}
P(\mathcal{E} \cap \{\text{no weak cycle in }H(k)\text{ spanning }V_1(H(0))) \leq \frac{k_0}{n^2} \to 0,
\end{equation*}
which completes the proof of the first part of the lemma.

$\mathbf{(ii)}.$ The argument for this part is effectively the same as before except we start with $H_d(n,p_1)$ and do $k_1$ steps rather than start with $H_d(n, p_0)$ and run $k_0$ steps. Because of this, we omit the proof.
\end{proof}

\section{Appendix}

\begin{lemma}\label{lemma: 10.6.1}
Suppose $|\omega| \leq \ln \ln n$ and $p=(d-1)! \frac{\ln n + \omega}{n^{d-1}}$. Then w.h.p. there is only one non-trivial component in $H_d(n,p)$.
\end{lemma}

\begin{proof}
By the hypergraph version of de la Vega's Theorem (Lemma \ref{hyper de la vega}), there is w.h.p. a component of size at least $n-B \frac{n}{\ln n}$ in $H_d(n,p)$ for some constant $B>0$. Therefore, it suffices to show that w.h.p. there are no components of size in $[d, B \frac{n}{\ln n}]$. We do this by showing the expected number of such components, $E_1$, tend to zero. In particular,
\begin{align*}
E_1 &= \sum_{A \subset [n], \, d \leq |A| \leq B \frac{n}{\ln n}} P(\text{component on vertex set }A) \\ &=\sum_{a=d}^{B \frac{n}{\ln n}} {n \choose a} P(\text{induced graph on }[a]\text{ is connected}, [a]\text{ is isolated from }[n] \setminus [a]).
\end{align*}
These latter no events are independent since they consider different groups of potential edges. Further, the induced hypergraph on $[a]$ is distributed as $H_d(a,p)$. Therefore,
\begin{equation*}
E_1 = \sum_{a=d}^{B \frac{n}{\ln n}} {n \choose a} P(H_d(a,p) \text{ is connected}) \, q^{{n \choose d}-{n-a \choose d}}.
\end{equation*}
Just as in \eqref{eqn: 10.5.1}, we have that
\begin{equation*}
P(H_d(a,p) \text{ is connected}) \leq \left( e \, a^{d-1} \, \frac{p}{(d-1)!} \right)^{\frac{a-1}{d-1}};
\end{equation*}
this bound holds as long as $a=o(n)$, which is definitely true in our case. Taking on the $q$-term, we see that
\begin{align*}
\log q^{{n \choose d}-{n-a \choose d}} & \leq -p \left( {n \choose d}-{n-a \choose d} \right) \\& = -p \, \frac{n^{d-1}}{(d-1)!} \left( a + O(a^2/n) \right) \leq  -a \ln n + a \ln \ln \ln n + O(a).
\end{align*}
Hence, for $n$ sufficiently large, we have that
\begin{equation*}
q^{{n \choose d}-{n-a \choose d}} \leq \exp \left( a (-\ln n + 2 \ln \ln \ln n) \right) = \left( \frac{(\ln \ln n)^2}{n} \right)^a.
\end{equation*}
Consequently, we have that
\begin{align*}
E_1 &\leq \sum_{a=d}^{B \frac{n}{\ln n}} \left( \frac{e \, n}{a} \right)^a \left(e \, a^{d-1} \frac{p}{(d-1)!} \right)^{\frac{a-1}{d-1}} \left( \frac{(\ln \ln n)^2}{n} \right)^a \\& = \sum_a \frac{1}{a} \left( e \, (\ln \ln n)^2 \right)^a \left( e \frac{p}{(d-1)!} \right)^\frac{a-1}{d-1}=: \sum \gamma_a.
\end{align*}
To show that this latter sum tends to zero, it suffices to show that $\gamma_d$ tends to zero and that the ratios of consecutive terms also (uniformly) tend to zero.

First, note that 
\begin{equation*}
\frac{\gamma_{a+1}}{\gamma_a} = \frac{a}{a+1} \left( e \, (\ln \ln n)^2 \right) \left(e \, \frac{p}{(d-1)!} \right)^{1/(d-1)} = O\left( \frac{ (\ln \ln n)^2 (\ln n)^{1/(d-1)}}{n} \right) \to 0;
\end{equation*}
and further that
\begin{equation*}
\gamma_{d} = O\left( (\ln \ln n)^{2d} \, \frac{\ln n}{n^{d-1}} \right) \to 0.
\end{equation*}

\end{proof}

\begin{remark}
A consequence of this lemma, along with Lemma \ref{lemma: converge in distribution}, is that the probability that $H_d(n,p)$, $p=(d-1)! \frac{\ln n + c}{n^{d-1}}$, is connected tends to $e^{-e^{-c}}$. 
\end{remark}

\end{document}